\documentclass{amsart}
\usepackage[margin=110pt]{geometry}
\usepackage{amssymb,amsmath,amsthm,epsf,epsfig,dsfont,bbm}
 
\usepackage{bm}
\usepackage{latexsym}
\usepackage{upref, eucal}
\usepackage[normalem]{ulem} 
\usepackage[all]{xy}
\usepackage{xcolor}
\usepackage{palatino} 
\usepackage{enumitem} 
\newcommand {\nc} {\newcommand}
\newcommand {\enm} {\ensuremath}

\def \d{\delta}
\def \uuu{u}
\nc {\bdm} {\begin{displaymath}}
\nc {\edm} {\end{displaymath}}

\numberwithin{equation}{section}
 
\newtheorem {theorem}[equation]{\bf{Theorem}} 
\newtheorem {lemma}[equation]{\bf Lemma}
\newtheorem {proposition}[equation]{\bf Proposition}
\newtheorem {corollary}[equation]{\bf Corollary}
\newtheorem {remark}[equation]{Remark}

\theoremstyle{definition}

\newtheorem {example}[equation]{Example}

\newcommand\fR{\mathfrak{R}} 
\renewcommand\AA{\mathbb{A}}
\newcommand\FF{\mathbb{F}}
\newcommand\WW{\mathbb{W}}
\newcommand{\Ou}{\enm{\mathcal{O}}}

\newcommand{\cR}{\mathcal{R}}
\nc {\form}[1] {\enm{\mbox{\underline{for}}}_{#1}}
\nc {\prol}[1] {\enm{\mbox{\underline{prol}}_{{#1}^*}}}

\nc {\stk} {\stackrel}

\newcommand{\map}{\rightarrow}

\newcommand{\Pn}[2] {\ensuremath{ {\mathbb{P}}^{#1}_{#2}}}
\nc{\Quot}[3]{\enm{ {\mathfrak{Quot}_{ {#1}/{#2}/{#3}}}}}
\nc{\Hilb}[2]{\enm{ {\mathfrak{Hilb}_{ {#1}/{#2}}}}}

\nc {\Coh}[4] {\ensuremath{H^{#1}(\Pn{#2}{},{#3}({#4}))}}
\nc {\Ch}[3] {\enm{H^{#1}(X_t,{#2}_t({#3}))}}
\nc {\Qphi}[4]{\enm{ {\mathfrak{Quot}^{~#4}_{ {#1}/{#2}/{#3}}}}}
\nc {\Gra}[4]{\enm{ {\mathfrak{Grass}_{#2}({#3},{#4})}}}
\nc {\HomA}[2]{\enm{\mathrm{Hom}_A{#1}{#2}}}
\nc {\tr}{\mathrm{tr}}
\DeclareMathOperator{\colim}{colim}

\nc {\C}[2]{\enm{\left(\begin{array}{l} {#1} \\ {#2} \end{array} \right)}}

\def \mb{\mbox}
\def \Z{{\mathbb Z}}

 \def \Z{{\mathbb Z}}

   \def \h{\hat{\ }}

\def \d{\delta}

\def \cO{\mathcal O}  \def \bX{{\bf X}} \def \bH{{\bf H}}

\def \R1{R((q))[q']\h}

\usepackage{xcolor}

\DeclareMathOperator{\Spec}{\mathrm{Spec}}
\DeclareMathOperator{\Spf}{\mathrm{Spf}}

\newcommand{\Hom}{\mathrm{Hom}}

\newcommand{\oldmarginpar}[1]{}

\nc{\bx}{\bold{x}}
\nc{\by}{\bold{y}}
\nc{\bz}{\bold{z}}
\nc{\ba}{\bold{a}}
\nc{\Fp}{\tilde{F}}
\nc{\Rp}{\tilde{R}}
\nc{\mlow}{m_{\mathrm{l}}}
\nc{\mup}{m_{\mathrm{u}}}
\nc{\ord}{\mb{ord }}
\nc{\bXp}{\bX_{\mathrm{prim}}}
\nc{\bPsi}{\bold{\Psi}}
\nc{\mult}{\mathrm{mult}}
\nc{\mbB}{\mathbbm{B}}
\nc{\mfor}[1]{{#1}^{\mathrm{for}}}
\nc{\Hdr}{\bH^1_{\mathrm{dR}}(A)}

\nc{\mP}{\mathfrak{p}}

\nc{\Proalg}[1]{\mathcal{C}_{#1}}
\nc{\prolarrow}[1]{\stk{(\uuu_{#1},\d_{#1})}{\longrightarrow}}
\nc{\tc}{\textcolor}
\nc{\nexp}[1]{\exp_{\d,{#1}}}
\nc{\gr}{\mathrm{Gr}}
\newcommand{\g}{\mathfrak{g}}

\newcommand{\Jbu}[1]{{\widehat{J^n_{\enm{Bu}}{{#1}}}}}
\newcommand{\Jbo}[1]{{J^n}{#1}}
\nc{\Sn}{S^{(n)}}
\nc{\Snr}[1]{S^{(n)}_{#1}}
\nc{\Sp}[1]{{\mathbf{Sp}_{#1}}}
 \nc{\Aff}[1]{{\mathbf{Aff}_{#1}}}
\nc{\Sch}[1]{{\mathbf{Sch}_{#1}}}
\nc{\undef}{{\color{red} (Undefined)}}
\nc{\ov}[1]{{\overline{#1}}}

\title{Arithmetic Jet Spaces}
\author{Alessandra Bertapelle}
\address{Universit\`a degli Studi di Padova, Dipartimento di Matematica ``T. Levi-Civita'', via Trieste 63, I-35121 Padova}
\email{alessandra.bertapelle@unipd.it}
\author{Emma Previato}
\address{Department of Mathematics and Statistics, 
Boston University, Boston, MA 02215-2411}
\email{ep@math.bu.edu}
\author{Arnab Saha}
\address{Indian Institute of Technology Gandhinagar, Gujarat 382355}
\email{arnab.saha@iitgn.ac.in}

\begin{document}

 \baselineskip 15pt
\begin{abstract}
	We extend Borger's construction of algebraic jet spaces to allow for	an arbitrary prolongation sequence, clarify the relation between 	Borger's and Buium's jet spaces and compare them in the extended 	sense. As a result, we strengthen a result of Buium on the	relation between Greenberg's transform and the special fiber of jet spaces,
	including the ramified case.
\end{abstract}
\maketitle

\section{Introduction}
In this paper we establish some of the foundational aspects of arithmetic jet space theory. 
In \cite{bui95} Buium introduced the general relative arithmetic jet spaces as prolongation sequences satisfying a universal property in terms of lifting $p$-derivations. It is noteworthy to mention that this theory was developed in the framework of $p$-adic formal geometry.
Since then it has been extensively studied and applied to diophantine geometry \cite{bui96}, \cite{BP}, differential 
modular forms \cite{bui00}, \cite{BuSa1}, \cite{BuSa2} and $p$-adic 
Hodge theory \cite{BoSa1}, \cite{BoSa2}. 
Very recently  the theory of $\delta$-rings (rings endowed with a $p$-derivation $\d$) have led to the development of the prismatic cohomology by  Bhatt and Scholze \cite{BhSc}. 

 On the other hand, over a Dedekind domain $\Ou$ with finite residue fields, Borger introduced arithmetic jet spaces in \cite[\S 10.3]{bor11b} which are purely algebraic. When $\Ou = \Z$, he showed that the completion of the Borger jet space along the special fiber over the fixed prime $p$ is in fact Buium's absolute arithmetic jet space.
 The word
absolute above refers to the fact that the objects and the adjunction property
are over $\Spec \Z$. However, the general setting of arithmetic jet spaces over any fixed prolongation sequence (which is not necessarily the constant prolongation sequence) is not considered in \cite{bor11b}.

In this paper, we generalize the above result in the relative setting where the constant prolongation sequence based on $\Spec \Z$ is replaced by an arbitrary prolongation sequence. 
To do so, our main result is the adjunction property proved in Theorem \ref{t.adjunction}, which we now describe in greater detail.

Firstly, we would like to fix some notations.
Let $\Ou$ be a Dedekind domain with $K$ as its field of fractions and fix a 
non-zero prime ideal $\mP \subset \Ou$. Let us also fix $\pi$ to be a generator of the prime ideal $\mP$. Let $k = \Ou/\mP$ be the residue field and also 
assume it to be a finite field of characteristic $p$ and cardinality $q=p^h$. 
The identity map on $\Ou$ is also tacitly fixed as the lift of the $q$-th power Frobenius, i.e., the identity on $k$, and the Witt vectors and jet spaces considered in this paper are regarding lifts of Frobenius compatible with respect to this choice of the identity map on $\Ou$.

The Witt vectors $W_n$ described in \cite{bor11a} are always considered with respect to the above data which we will denote by the pair $(\Ou,\mP)$. In particular 
the most commonly used $p$-typical Witt vectors come from considering $\Ou= \Z$ and $\mP=(p)$, for some prime $p \in \Z$, 
It is noteworthy to mention that the Witt vector functor $W_n$ depends on the
choice of the {\it base ring} $\Ou$ and the maximal ideal $\mP$.

Given $\Ou$-algebras $A$ and $B$, we define a {\it prolongation} $A \prolarrow{} B$ to be the data  of two maps $\uuu,\delta$,
where $\uuu\colon A \map B$ is a morphism of $\Ou$-algebras and $\d\colon A \map B$ is a set-theoretic map (also known as a {\it $\pi$-derivation relative to} $\uuu$)
that satisfies:
\begin{eqnarray}	\d (x + y) &=& \d( x) + \d(y) + C_\pi(\uuu(x), \uuu(y)) ,\label{d.sum}\\
	 \d (xy) &=& \uuu(x)^q\d (y) + \uuu(y)^q \d( x) + \pi \d( x) \d( y)  ,\label{d.prod}
\end{eqnarray} 
for all $x,y \in A$,
where $C_\pi(X,Y) = \frac{X^q + Y^q -(X+Y)^q}{\pi}\in \Ou[X,Y]$.

A sequence of prolongations of $\Ou$-algebras 
$\{C_{0} \prolarrow{0} C_{1}\prolarrow{1} C_{2} \dots\}$ is called a 
{\it prolongation sequence} if for each $n\geq 0$, 
\begin{align*}
\uuu_{n+1} \circ \d_n = \d_{n+1} \circ \uuu_n;
\end{align*}
we will denote it by  
$C_*:=\{C_n\}_{n=0}^\infty$. 
As an example, consider the   prolongation sequence $C_i=\Ou, u_i=\mathrm{id}, \d_i(x)=\d(x)=\pi^{-1}(x-x^q)$.
It will be denoted  $\Ou_*$ in what follows.
Prolongation sequences naturally form a category.

Let $R_*$ be a fixed prolongation sequence over $\cO_*$ and let $\Proalg{R_*}$ denote
the category of prolongation sequences over $R_*$.
Following  \cite{bui00} l we consider the functor  
\[J_*: \mb{Alg}_{R_0} \map \Proalg{R_*}, \quad A\mapsto J_*(A) = 
\{J_n(A)\}_{n=0}^\infty\] 
which associates to any $R_0$-algebra the canonical prolongation sequence of its $\pi$-jet algebras.

Then as in Proposition $1.1$ in \cite{bui00} $J_*A$ satisfies the following 
universal property:
\[
\Hom_{R_0}(A,C_0) \simeq \Hom_{\Proalg{R_*}}(J_*(A),C_*)
\]
for any prolongation sequence $C_* \in \Proalg{R_*}$.
Our  main  result is the following
natural bijection
\[\Hom_{R_0}(A,W_n(C)) \simeq \Hom_{R_n}(J_n(A),C),
\]
for any $R_0$-algebra $A$ and any $R_n$-algebra $C$, i.e., 
the functors $J_n$  and $W_n$ are adjoint; see Theorem \ref{t.adjunction}.	 
To prove the above result, we use an explicit expression (\ref{eq.expcomp}), an inductive formula that provides a correspondence between the Witt and the Buium-Joyal delta coordinates.
Note that here we adopt Borger's notation shifting indices on Witt vectors; see Remark \ref{rem.wn}.  The adjointness of $J_n$ and $W_n$ allows us to prove an explicit comparison result between Buium's $\pi$-jet spaces and Borger's one, even in our more general context (Corollaries \ref{BoBua}, \ref{cor.nm}).

In the last part of the paper we compare Greenberg's realization of a scheme $X$ defined over a discrete valuation ring with the special fiber of the algebraic jet space associated to $X$; see Theorem \ref{t.grj}. In particular we show that the special fiber of the $p$-jet space of a scheme $X$ defined over a complete unramified extension of $\Z_p$ coincides with the Greenberg realization of $X$; see Corollary \ref{c.infty}. This generalizes \cite[Theorem 2.10]{bui96} removing the smoothness hypothesis and a condition on the residue field.

{\bf Acknowledgements. } The authors wish to thank James Borger for several 
insightful discussions as well as bringing the authors together for 
collaboration. The second author is deeply indebted to the Boston University--Universit\`a degli Studi di Padova Faculty Exchange Program for the support that made this work possible.
The third author is also grateful to the Max Planck Institute for Mathematics 
in Bonn for its hospitality and financial support.

\section{Prolongation sequence algebras}

\subsection{$\pi$-derivations}
We look more closely    at $\pi$-derivations and prolongation sequences. Let $A,B$ be  $\cO$-algebras, $ \rho_A\colon \cO\to A, \rho_B\colon \cO\to B$  the structure morphisms and $\uuu\colon A\to B$ a morphism of $\cO$-algebras. 
Recall that a {\it $\pi$-derivation relative to $\uuu$} is a map $\d_A=\d
\colon A\to B$ that satisfies conditions \eqref{d.sum} and \eqref{d.prod}.
In the following we will  sometimes write $\d x $ in place of $\d(x)$ to relax 
the notation.

 The ring $\Ou$ is endowed with a 
unique $\pi$-derivation $\delta(x)=\pi^{-1}(x-x^q)$. We will say that $A$ is an {\it $\Ou$-algebra with $\pi$-derivation } if $A$ is endowed with a $\pi$-derivation $\d_A$ relative to ${\rm id}_A\colon A\to A$ such that $\d_A\circ \rho_A=\rho_A\circ \d  $.  
 Therefore, if $A$  is  an $\Ou$-algebra  such that $\pi A=0$, i.e.,  a $k$-algebra, $A$ admits no structure of   $\cO$-algebra with $\pi$-derivation. 
 On the other hand, if $A$  has no $\pi$-torsion, any  $\cO$-algebra endomorphism $\phi\colon A\to A$ which lifts the $q$-Frobenius modulo $\pi A$ produces a $\pi$-derivation  $\d_A (x)=\pi^{-1}(\phi(x)-x^q)$ so that $A$ has a structure of $\cO$-algebra with $\pi$-derivation.
 Morphisms between $\cO$-algebras with $\pi$-derivation are defined in 
the obvious way.

Given $\Ou$-algebras   $A$ and $B$, we define a {\it prolongation} to be the data $(\uuu,\delta)$ of two maps
\[\xymatrix{A \ar@<0.5ex>[r]^{\uuu}
	\ar@<-0.5ex>[r]_\delta 
	&B} 
\]	
where $\uuu\colon A \map B$ is a morphism of $\Ou$-algebras and $\d\colon A \map B$ is a   $\pi$-derivation relative to $\uuu$.
For brevity, we will denote the above datum by \[A \prolarrow{} B.
\]
Prolongations form a category where morphisms are given by pairs of  homomorphisms $(f,g)$ of $\cO$-algebras making the obvious squares 
\[\xymatrix{A \ar[d]^f\ar@<0.5ex>[r]^{\uuu}
	\ar@<-0.5ex>[r]_\delta 
	&B   \ar[d]^g\\
	A' \ar@<0.5ex>[r]^{v}
	\ar@<-0.5ex>[r]_\delta& B'
}
\]
commute, i.e., $g\circ \uuu=v\circ f$, $g\circ\d=\d\circ f$. 
Note that for $B=A, B'=A', f=g$, $u$ and $v$ the identity homomorphisms,  the above diagram denotes a morphism of $\cO$-algebras with $\pi$-derivations.

Let   $\{C_n, u_n, \d_n\}_{n=0}^\infty$ or simply  
$C_*:=\{C_n\}_{n=0}^\infty$ denote a {\it prolongation sequence} of $\cO$-algebras, i.e. a sequence of prolongations of $\Ou$-algebras 
$\{C_{0} \prolarrow{0} C_{1}\prolarrow{1} C_{2} \dots\}$ such that for each $n\geq 0$, 
\begin{align}\label{comprol}
\uuu_{n+1} \circ \d_n = \d_{n+1} \circ \uuu_n.
\end{align}
We will denote as $\Ou_*$ the   prolongation sequence $C_i=\Ou, u_i=\mathrm{id}, 
\d_i(x)=\d(x)=\pi^{-1}(x-x^q)$. 

A morphism $f_*\colon C_*\to D_*$ between two prolongation sequences is a sequence of $\cO$-algebra homomorphisms $f_n\colon C_n\to D_n$ which induces homomorphisms of prolongations
\[\xymatrix{C_n \ar[d]^{f_n}\ar@<0.5ex>[r]^{\uuu_n}
	\ar@<-0.5ex>[r]_{\delta_n} 
	&C_{n+1}   \ar[d]^{f_{n+1}}\\
	D_n \ar@<0.5ex>[r]^{v_n}
	\ar@<-0.5ex>[r]_{\delta_n} & D_{n+1} 
}
\]
If $f_*$ exists, we will say that $D_*$ is over $C_*$.
It is immediate to check that prolongation sequences  form a category. Note that if $(\Ou,\mP)= (\Z,(p))$, then $\Ou_*$ is the initial object in the
category of prolongation sequences.
We denote by $\Proalg{R_*}$ the category of prolongation sequences over the fixed  prolongation
sequence $R_*$.

 In this paper we will only consider  prolongation sequences $C_*$ 
over $\cO_*$. 

Note that  we may think of these
 sequences as shifted prolongation sequences \[\{C_{-1}=\cO\stackrel{(\rho_{C_0}, \rho_{C_0}\circ \d)}{\longrightarrow} C_0  \prolarrow{0} C_{1}\prolarrow{1} C_{2}\dots\}.
	\]

Now we introduce  polynomials $Q^\d$ whose role will be  more clear  after the introduction of $\pi$-jet algebras; see Remark \ref{rem.q}. 
In fact $Q^\d$ will be the image of $Q$ via the universal $\pi$-derivation. 
Following  Buium's notation, we will work with indeterminates $T=T^{(0)}, T'=T^{(1)}, T''=T^{(2)}, \dots$
 \begin{proposition}\label{p.expd}
Consider the polynomial ring   $A=\Ou[T ,T',T'',\dots ]$ and the $\cO$-linear endomorphism $\phi_A\colon A\to A$ that maps $T^{(i)}$ to $(T^{(i)})^q+\pi 
T^{(i+1)}$. It  lifts   the $q$-Frobenius  on $k[T ,T',\dots]$. Let $\d_A$ be the corresponding $\pi$-derivation and let $Q^\d$ denote $\d_A(Q)$ for any polynomial $Q\in \Ou[T ,T',\dots ]$. 
\begin{itemize}
	\item[i)] If $Q=T^{(i)}$ then $Q^\d=T^{(i+1)}$.
	\item[ii)] If $Q\in \Ou[T ,\dots, T^{(n)} ]$ then $Q^\d\in  \Ou[T ,\dots, T^{(n+1)}]$.
	\item[iii)] For any   $\Ou$-algebra with $\pi$-derivation $(B,\d)$
 and any $b\in B$ it is
	\[\d \left(Q(b,\d b, \dots,\d^{n} b)\right)=Q^\d(b,\d b,\dots,\d^{n+1}b)\in B
	\] 
	 where $\d^n\colon B\to B$ denotes the $n$-fold composition of $\d$.
\end{itemize}
\end{proposition}	
\begin{proof}
Assertions i) and ii) are immediate since $\d_A(Q)=\pi^{-1}(\phi_A(Q)-Q^q)$. 
For iii), let  $\phi\colon B\to B$ be the   ring endomorphism   lifting the $q$-Frobenius on $B/\pi B$ and associated to $\d$, i.e., $\phi(b)=b^q+\pi\d (b)$ for any $b\in B$. For a fixed $b\in B$, let  $f_b\colon A\to B$ denote  the 
morphism of $\cO$-algebras mapping $T^{(i)}$ to $\d^i(b)$. Since 
\[\phi(f_b(T_i))=\d^i(b)^q+\pi \d^{i+1}(b)=f_b\left(({T^{(i)}})^q+\pi T^{(i+1)}\right)=f_b(\phi_A(T^{(i)})),
\]  $f_b$ is a morphism of $\cO$-algebras with $\pi$-derivation, i.e., $\d \circ f_b=f_b\circ \d_A$. Hence
\[\d \left(Q(b,\d b, \dots,\d^{n} b)\right)=\d(  f_b(Q))=f_b(\d_A(Q)) =Q^\d(b,\d b,\dots,\d^{n+1}b),  \]
as desired. 
\end{proof}

\subsection{$\pi$-typical Witt vectors}\label{s.pityp}
With the fixed $\cO$, $\pi$ and $q$ as above, we define the $\pi$-typical Witt vector functors $W_n=W_{\pi,q,n}$ (where we omit the indices $\pi,q$ if clear from the context) in terms of the Witt polynomials; see  \cite{haz,sch} for $\cO$ a  discrete valuation ring, \cite{bor11a} for $\cO$ a Dedekind domain. 
See also \cite[\S 1.1]{ahs} for the ramification ring structure $(\cO, \pi, q)$.

For each $n \geq 0$ let us define the  $n$th \emph{ghost ring} $\prod_{i=0}^n B$  to be the $(n+1)$-product $\Ou$-algebra $ B  \times \cdots  \times B$ where 
$\lambda(b_0,\dots, b_n)=(\lambda b_0,\dots, \lambda b_n)$ for any $\lambda\in \Ou$, and similarly for the infinite product $\Pi^\infty_{i=0} B:= B \times B \times \cdots$.
Then for all $n \geq 1$ there exists a \emph{restriction}, or \emph{truncation}, map \[T_w\colon \Pi^n_{i=0} B \map \Pi^{n-1}_{i=0}B, \qquad (b_0,\cdots,b_n)\mapsto (b_0,\cdots,b_{n-1}),
\]
which is a  morphism of $\Ou$-algebras.

Now define $ W_n(B)=B^{n+1}$ as sets, and define the map of sets 
\[ w\colon W_n(B) \map \Pi_{i=0}^n B,\qquad  b_.=(b_0,\dots,b_n)\mapsto ( w_0(b_.),\dots,w_n(b_.)) 
\] where
\[ 
w_i = x_0^{q^i}+ \pi x_1^{q^{i-1}}+ \cdots + \pi^i x_i
\]
are the so-called \emph{Witt polynomials} or ghost polynomials.
The map $w$ is known as the {\it ghost} map. 

\begin{remark}\label{rem.wn}
	Do note that under the traditional indexing our $W_n$ would be denoted $W_{n+1}$; our choice of following Borger's non-standard notation is motivated by the adjunction formula   \eqref{eq.adj0}.
\end{remark}

The ghost map $w$ is injective when $B$ has no $\pi$-torsion, e.g., if $B$ is a ring of polynomials over $\Ou$ with possibly infinitely many indeterminates. This implies that  we  can endow $W_n(B)$ with a unique   $\Ou$-algebra structure such that $w$ becomes a morphism of $\Ou$-algebras.
On the other hand, for general $B$, we can induce such $\Ou$-algebra structure by representing $B$ as quotient of $\Ou[x_i]$ for a suitable family of indeterminates $x_i$;
cf. \cite[Ch 1. \S 1.1]{ff}, \cite[Proposition 1.1.8]{sch}, \cite[\S 2.4]{bor11a}. This construction is functorial in $B$.

 We will refer to $W_n(B)$ as the ring of {\it truncated $\pi$-typical Witt vectors}.  When $\Ou$ is a discrete valuation ring it is usually called the ring of truncated ramified Witt vectors.

There are \emph{restriction}, or \emph{truncation}, maps 
\[T\colon W_n(B) \map W_{n-1}(B), \qquad (x_0,\dots,x_n) \mapsto  (x_0,\dots, x_{n-1}),
\]    
which satisfy $T_w\circ w_i=w_i\circ T$  for $i\geq 1$. 	
We define the ring of \emph{$\pi$-typical Witt vectors} $W_{q,\pi}(B)=W(B) = \varprojlim W_n(B)$ with restrictions $T$ as transition maps, where subscripts $q,\pi$ are usually omitted since clear from the context.

\begin{remark}
	Given $\cO$-algebras $A$ and $B$ any prolongation $(\uuu,\d)$ of $A$ to $B$ determines a morphism of $\cO$-algebras $ A\to W_1(B), a\mapsto (\uuu(a),\d(a))$,  and conversely (cf. \cite[1.2]{bui95}).
\end{remark}

\subsection{Operations on Witt vectors}
\label{subsec-witt-operations}
Now we recall some important functorial operators on the Witt vectors. 
They are constructed from operators on the  ghost rings defined above.

First consider the left shift \emph{Frobenius} operator  on the ghost rings 
\[
F_w\colon \Pi^n_{i=0} B \map \Pi^{n-1}_{i=0} B, \qquad (b_0,\dots,b_n) \mapsto (b_1,\dots,b_n). 
\]
and passing to limit
\[
F_w\colon \Pi^\infty_{i=0} B \map \Pi^{\infty}_{i=0} B, \qquad (b_0,b_1,\dots) \mapsto (b_1,b_2,\dots). 
\] 
They are morphisms of $\cO$-algebras.
The {\it Frobenius}   morphism 
$F\colon W_n(B) \map W_{n-1}(B)$, is the unique map which is functorial in $B$ and makes the following diagram 
\begin{align*}
\xymatrix{
	W_n(B) \ar[r]^w \ar[d]_F & \Pi^n_0 B \ar[d]^{F_w} \\
	W_{n-1}(B) \ar[r]_-w & \Pi_0^{n-1} B
}  
\end{align*}
commute. It is a morphism of $\cO$-algebras. 	Similarly one defines the Frobenius morphism $F\colon W(B) \map W(B)$.  

Next we consider the $\cO$-linear map
\[V_w\colon \Pi_0^{n-1}B \map \Pi_0^n B, \qquad V_w( b_0,\dots ,b_{n-1})\mapsto (0, \pi b_0,\dots,\pi b_{n-1}). 
\]
The {\it Verschiebung} $V\colon W_{n-1}(B) \map W_n(B)$ is the map given by
\[V(x_0,\dots,x_{n-1}) = (0,x_0,\dots,x_{n-1});
\] 
it makes the following diagram commute:
\begin{align*}
\xymatrix{
	W_{n-1}(B) \ar[r]^-w \ar[d]_V & \Pi_0^{n-1}B \ar[d]^{V_w} \\
	W_n(B) \ar[r]_-w & \Pi_0^nB
} 
\end{align*}
Similarly one defines the Verschiebung $V\colon W(B) \map W(B)$. The Frobenius 
and the Verschiebung satisfy the identities
\begin{align} 
\label{FV-pi}
FV(x) = \pi x \quad \forall x\in W(B);\\
\label{FV2}
x\cdot V(y)=V(F(x)\cdot y) \quad \forall x,y\in W(B).
\end{align}

The Verschiebung is not a ring homomorphism, but it is $\Ou$-linear. Indeed, this is clear when $B$ has no $\pi$-torsion since $V_w$ and $w$ are morphisms of $\Ou$-algebras; the general case follows considering any surjective homomorphism $\Ou[x_i, i\in I]\to B$.

Finally, we have the multiplicative {\it Teichm\"uller map} $[~]:B \map W(B)$ given by
$x\mapsto [x]= (x,0,0,\dots)$. 

\begin{remark}\label{r.delta}
	The Frobenius on $W(B)$ lifts the  $q$-Frobenius on $W(B)/\pi W(B)$, and hence it induces a $\pi$-derivation $\Delta$ on $W(B)$ relative to the identity. Although this is a well-known fact, we demonstrate by using the previous identities. Let $x=(x_0,x_1,\dots)\in W(B)$ and write $x=[x_0]+V(y)$ with $V $ the Verschiebung map and $y=(x_1,x_2,\dots)$. Then by \eqref{FV-pi}
	\[F(x)= F[x_0]+FV(y)  \equiv [x_0^q]\quad  (\text{\rm mod } \pi W(B)). \]
	On the other hand,
	\[x^q=([x_0]+V(y))^q=[x_0^q]+V(y)^q +p(\dots)\equiv  [x_0^q]+V(y)^q 
	\equiv  [x_0^q]\quad (\text{\rm  mod  } \pi W(B)), \]
	where the first equality is due to the fact that  $p$ divides ${q}\choose{j}$ for $0< j<q$, the equivalence in the middle follows from $p\in  \pi\cO$ and the latter equivalence  follows from \eqref{FV-pi} and \eqref{FV2} writing $V(y)^{2}=V(FV(y)\cdot y)=V(\pi y^2)=\pi V(y^2)$ by $\cO$-linearity of $V$. One concludes then that $F(x)-x^q \in \pi W(B)$.  
\end{remark}

Now we recall the universal property of $\pi$-typical Witt vectors: 
\begin{theorem}
	\label{uniprol}
	The functor $W$ is	the right adjoint of the forgetful functor   $ \mathrm{Alg}_\d \to  \mathrm{Alg}$ from  the category of $\Ou$-algebras with $\pi$-derivation to the category of $\Ou$-algebras, i.e., there is a natural bijection 
	\begin{equation*} 
	\Hom_{\mathrm{Alg}_\d}(C,W(B)) \simeq \Hom_{  \mathrm{Alg}}(C,B), 
	\end{equation*}
	for any $\Ou$-algebra $B$ and any $\Ou$-algebras with $\pi$-derivation $C$. 
\end{theorem} 

This fact is explained in \cite{joyal}  for the case $\pi=p$  and in \cite[\S 1]{bor11a} for the general case. Here we review the theory as follows:

The ring $W(B)$ of  $\pi$-typical Witt vectors  is  the unique (up to unique isomorphism) $\Ou$-algebra $W(B)$ with a $\pi$-derivation $\Delta$ on $W(B)$ relative to the identity (see Remark \ref{r.delta}) and an $\Ou$-algebra restriction homomorphism $T\colon W(B) \map B$ such that, given any
$\Ou$-algebra $C$ with a $\pi$-derivation $\d$ relative to ${\rm id}_C$ and an 
$\Ou$-algebra morphism $f\colon C \map B$, there exists a unique $\Ou$-algebra 
morphism $g\colon C \map W(B)$  such that the diagram
\[
\xymatrix{
	& W(B) \ar[d]^T  \\
	C \ar[r]_f \ar[ur]^g & B
}
\]
commutes and $g \circ \d = \Delta \circ g$.

Note that if $B$ has a $\pi$-derivation $\d $ relative to the identity, the above theorem asserts the existence of a canonical section $\exp_\d\colon B\to W(B)$ of $T$;  one can check that $\exp_\d$ makes the diagram
\begin{equation}\label{d.exp}
\xymatrix{
	W(B) \ar[r]_w & \prod_{i=0}^nB\\
	B\ar[u]^{\exp_\d} \ar[ur]_{  ( {\rm id},\phi,\phi^2,\dots)}&
}
\end{equation}
commute, where  $\phi\colon B\to B$ maps $b$ to $b^q+\pi\d(b)$; cf. \cite[Prop. 1.1.23]{sch}. In particular $\exp_\d(b)=(b, \d (b), \dots)$ and the other entries are polynomials in $\d^i(b)=\d(\d(\dots\d(b)))$ with coefficients in $B$.

In the general case, one defines $g$ as the composition of the natural section $\exp_\d$ of $T\colon W(C)\to C$ with $W(f)$:
\[
\xymatrix{
	W(C) \ar@/_/[d]_T\ar[rr]^{W(f)}  & & W(B)  \ar[d]^T \\
	C \ar[rru]_g \ar[u]_{\exp_\d} \ar[rr]_f& & B
}
\]
Thanks to Proposition \ref{p.expd} we can make more explicit the description of $\exp_\d$ in \eqref{d.exp}:
\begin{proposition}	\label{p.pp} 
	There exist polynomials $P_n\in \Ou[T,T'\dots T^{(n)}]$ such that for any $\Ou$-algebra with $\pi$-derivation $B$ and any $b\in B$
	\[\exp_\d(b)= (P_0(b), P_1(b),\dots),
	\]    where $P_n(b)=P_n(b,\d b,\dots, \d^n b)$. Further $P_0(T)=T, 
	P_1(T,T')=T'$ and for every $n\geq 1$
	\[ P_n  =  P_{n-1}^\d + \sum_{i=0}^{n-2}\left[\sum_{j=1}^{q^{n-1-i}}
	\pi^{i+j-n} \C{q^{n-1-i}}{j} P_i^{ q(q^{n-1-i}-j)}( P_i^\d)^j\right],\]
	with $ P^\d_i(b)=\d(P_i(b,\d b,\dots,\d^i b))$.
\end{proposition}
\begin{proof} From diagram \eqref{d.exp} we get $P_0(b)=b, P_1(b)=\d(b)$, and more generally
	
	\begin{eqnarray*}
		\sum_{i=0}^n \pi^i P_i(b)^{q^{n-i}} &=& \phi^n(b) \\
		&=& \phi (\phi^{n-1}(b)) \\
		&=& \phi\left( \sum_{i=0}^{n-1}\pi^i P_i(b)^{q^{n-1-i}} \right)\\
		&=& \sum_{i=0}^{n-1} \pi^i (\phi(P_i(b))^{q^{n-1-i}}\\
		&=& \sum_{i=0}^{n-1} \pi^i\left(P_i(b)^q + \pi \d(P_i(b))\right)^{q^{n-1-i}}\\
		&=& \sum_{i=0}^{n-1} \pi^i \left[ P_i(b)^{q^{n-i}} + \sum_{j=1}^{q^{n-1-i}}
		\C{q^{n-1-i}}{j}P_i(b)^{(q^{n-1-i}-j)q}\pi^j(\d P_i(b))^j \right] \\
		&=& \sum_{i=0}^{n-1}\pi^i P_i(b)^{q^{n-i}}+ \sum_{i=0}^{n-1} \left[
		\sum_{j=1}^{q^{n-1-i}}\pi^{i+j} \C{q^{n-1-i}}{j} P_i(b)^{(q^{n-1-i}-j)q}(\d P_i(b))^j\right]
	\end{eqnarray*}
	Hence cancelling the common terms on both sides of the above equality we get
	\[
	\pi^nP_n(b) = \pi^n\d P_{n-1}(b) + \sum_{i=0}^{n-2}\left[\sum_{j=1}^{q^{n-1-i}}
	\pi^{i+j} \C{q^{n-1-i}}{j} P_i(b)^{ q(q^{n-1-i}-j)}(\d P_i(b))^j\right].
	\]
	It is then clear that the polynomial
	\[
	P_n  =  P_{n-1}^\d + \sum_{i=0}^{n-2}\left[\sum_{j=1}^{q^{n-1-i}} c_{i,j}
	P_i^{ q(q^{n-1-i}-j)}( P_i^\d)^j\right]\quad \text{with}\quad c_{i,j}=\pi^{i+j-n} \C{q^{n-1-i}}{j}
	\]
	does the job since  $c_{i,j}\in  \Ou$. We would like to add a few words 
	explaining why $c_{i,j}$ lies in $\cO$ (although it is expected due to
	functorial reasons).   If $i+j-n\geq 0$ the fact is clear. Now assuming
	$n-i>j$, we will show that the $p$-adic valuation of $c_{i,j}$ is non-negative.  Let $q=p^h$ and let $e$ denote the ramification of $p$ in $\cO$. Since $i+j-n$ is negative and $n-1-i$ is positive, we have
	\[v_p(c_{i,j})=\frac{i+j-n}{e}+h(n-1-i)-v_p(j)> (i+j-n) +(n-1-i)-v_p(j)=j-1-v_p(j)>0,\]
	where we have used the formula \[
	v_p \C{p^{m}}{j}=m-v_p(j) .
	\]	
\end{proof}

In the case when $\pi=p, q=p$ the map $\exp_\d$ in \eqref{d.exp} is the one 
in \cite[Proposition 1]{joyal} and hence the polynomials $P^\d$ describe 
Joyal's $\d_n$-operation.

From now until the end of the section  let $R_*$ denote a  fixed prolongation sequence over $\cO_*$. 

\begin{proposition}
	\label{proluniv}
	Given any prolongation sequence $C_* \in \Proalg{R_*}$, for each $n$ there exists a canonical morphisms of $\Ou$-algebras \[C_0 \map W_n(C_n).\]
\end{proposition}
\begin{proof}
	Consider the prolongation sequence $C_*$ written as
	\[ C_0 \prolarrow{0} C_1 \prolarrow{1} \cdots ,\]
	and consider the $\Ou$-algebra obtained by taking the direct limit with respect to $\uuu_n$, denoted $C_\infty = \varinjlim C_n$. Thanks to the maps in
	\eqref{comprol}   
	the $\d_n$'s induce a $\pi$-derivation $\d$ on $C_\infty$ relative to the identity. 
	Then by the universal 
	property of Witt vectors in Theorem \ref{uniprol} we have a canonical morphism
	of $\Ou$-algebras
	\[\exp_\d: C_\infty \map W(C_\infty)\]
	where $\exp_\d(c) = (P_0(c), P_1(c),\dots)$  with $P_n(c):=P_n(c,\d c,\dots, \d^n c)$ and $P_n $ is the  polynomial introduced in  Proposition \ref{p.pp}.
	Therefore by composition we obtain 
	\[
	C_0 \stk{\uuu}{\map} C_\infty \stk{\exp_\d}{\longrightarrow}
	W(C_\infty) \stk{T}{\map} W_n(C_\infty)
	\] 
	where $T$ is the restriction map of Witt vectors.
	Now note that the map $C_0 \map W_n(C_\infty)$ is given by $c \mapsto
	(P_0(c),\dots, P_n(c))$ for all $c \in C_0$. Since $P_i(c)$, for each $i\leq n$, is a polynomial in $c,\dots, \d^ic$, with coefficients in $\Ou$, the map $T\circ~\exp_\d \circ~ \uuu$ factors through $W_n(C_n)$ and we are done.
\end{proof}

The above results say that there are, in general, no sections of the 
reduction map $W_n(C_0)\to C_0$. However one section exists by replacing $C_0$ 
with $C_\infty$ and the latter section induces a lifting of $C_0\to C_n$ to $W_n(C_n)$ as depicted below
\begin{equation*}
\xymatrix{ W_n(C_0)\ar[d]\ar@{}[drr]|(0.7){\circlearrowright}\ar[rr]^{W_n(\uuu)}&&W_n(C_n)\ar[d] \ar[rr] &&W_n(C_\infty)\ar[d]
	\\
	C_0\ar[rr]_{\uuu}\ar@{.>}[urr]^{\exists}&&C_n\ar[rr] &&C_\infty\ar@/_1.0pc/[u]_{T\circ \exp_\d} }\end{equation*}

\begin{corollary}
	\label{proluniv3}
	Let $B$ be an $R_n$-algebra. Then $W_n(B)$ is an $R_0$-algebra.
\end{corollary}
\begin{proof}
	Let $f\colon R_n \map B$ be the structure map. Now by Proposition \ref{proluniv}, we have the $\Ou$-algebra map $R_0 \map  W_n(R_n)$ and hence composition
	with $W_n(f)\colon W_n(R_n) \map W_n(B)$ gives us the claim.
\end{proof}

\subsection{$\pi$-jet algebras}
We continue with the notations of the 
previous subsection and let $R_*$ denote a  fixed prolongation sequence of 
$\Ou$-algebras.  We now consider the functor  
\[J_*: \mb{Alg}_{R_0} \map \Proalg{R_*}, \quad A\mapsto J_*A = 
\{J_nA\}_{n=0}^\infty\]
as defined by Buium in \cite{bui00}.  Assume $A$ has presentation $A= R_0[\bx]/(\bf{f})$ where $\bx$ is a collection (possibly infinite) of variables $x_\gamma, \gamma\in \Gamma$, and ${\bf f}$ is the system of defining polynomials ${\bf f} = (f_j)_{j \in I}$ where $f_j \in R[\bx]$ and $I$ is an indexing set. 
One has a prolongation sequence 
\begin{equation}\label{rrr}
R_0[\bx]\prolarrow{0} R_1[\bx,\bx^{(1)}]\prolarrow{1}R_2[\bx,\bx^{(1)},\bx^{(2)}] \dots\end{equation}
where $\bx^{(i)}$ is a set of indeterminates $x^{(i)}_\gamma, \gamma\in \Gamma$, $u_i$ is the obvious extension of $R_i\to R_{i+1}$ mapping $x^{(r)}_\gamma$, $0\leq r\leq i$, to itself and $\d_i$ is a $\pi$-derivation that extends the one on $R_i$ and  maps $x^{(r)}_\gamma$ to $x^{(r+1)}_\gamma$ for $0\leq r<i$. Then as in \cite{bui00} we set for    each $n$ 
\begin{align}
\label{JnA}
J_nA = R_n[\bx,\dots {\bx^{(n)}}]/({\bf f},\dots,  \d^n {\bf f}).
\end{align}
where   $\d^i {\bf f} = (\d^i f_j)_{j \in I}$; we obtain in this way a 
prolongation sequence
\[A=J_0A\prolarrow{0} J_1A \prolarrow{1} J_2A \dots\] where $\uuu_i.\d_i$ 
still denote the maps induced by the $\uuu_i$ and $\d_i$ in \eqref{rrr}. We 
further have ring homomorphisms
\[\phi=\phi_n\colon J_{n}A\to J_{n+1}A,\quad a\mapsto \uuu_n(a)^q+\pi\d_n(a),\]
and with $\phi^n\colon A\to J_nA$ we will denote the composition $\phi_{n-1}\circ\phi_{n-2}\circ\dots\circ \phi_0$.

For any  $R_0$-algebra  $A$, $J_*A = \{J_nA\}_{n=0}^\infty$ is  called the {\it  canonical prolongation sequence} of $\pi$-jet algebras over $R_*$ associated to $A$. Then for each $n$, $J_nA$ is an $R_n$-algebra and we can  consider the direct limit $J_\infty A= \varinjlim J_nA$, with transition maps $\uuu_n$, which is endowed with a canonical $\pi$-derivation $\d= \varinjlim  \d_n$ relative to the identity which is  $\varinjlim \uuu_n$.
Moreover, any $J_n A$ is an $A$-algebra via the $\uuu_i$; for any $a\in A$, we let the same letter denote its image in $J_n A$.

\begin{remark}\label{rem.q}
	For $A=\cO[T]$, then $J_nA=\cO[T,T',\dots, T^{(n)}]$ with $\uuu_n $ the obvious inclusion and $\d_n$ the map sending a polynomial $Q$ to $Q^\d$, as defined in Proposition \ref{p.expd}.
\end{remark}

Consider the canonical morphism $\nexp{n}\colon A \map W_n(J_nA)$ as constructed in Proposition  \ref{proluniv}; it is  given by
\begin{align}
\label{eq.expcomp}
\nexp{n}(a)=(P_0(a),\dots P_n(a))
\end{align}
where $P_i(a)$ are polynomials in $a,\d(a),\dots, \d^i(a)$ with coefficients in $\Ou$ as introduced in Proposition \ref{p.pp}. In particular we have the following recursive formula 
\begin{equation}\label{eq.expfor}
P_n(a) =  P^\d_{n-1}(a) + \sum_{i=0}^{n-2}\left[\sum_{j=1}^{q^{n-1-i}}
\pi^{i+j-n}  \C{q^{n-1-i}}{j}  P_i(a)^{ q(q^{n-1-i}-j)}(  P^\d_i(a))^j\right],
\end{equation} 
for any $n\geq 1$ and $a\in A$, where $ P^\d_i(a)=\d(P_i(a,\d(a),\dots,\d^i(a)))$.

Note that for $A=R_0[\bx]=R_0[x_\gamma,\gamma\in \Gamma]$, inductively applying \eqref{eq.expfor}, we have
\begin{equation} \label{eq.pxs}
P_n(x_\gamma)=x^{(n)}_\gamma+S_{n-1}(x_\gamma) \quad \text{with}\quad S_{n-1}(x_\gamma)\in R_n[x_\gamma,\dots, x^{(n-1)}_\gamma], \quad n\geq 1,\ S_0(x_\gamma)=0.
\end{equation}
In particular $J_n(R_0[\bx])=R_n[\bx,\dots, \bx^{(n)} ]=R_n[ \bx , P_1(\bx),\dots, P_n(\bx)]$, where $P_i(\bx)$ means the collection of the polynomials $P_i(x_\gamma),\gamma\in \Gamma$. More generally for any polynomial $f\in R_0[\bx]$ we have
\begin{equation}\label{eq.pfs}
P_n(f )=\d^{n}f +S_{n-1}(f ) \quad \text{with}\quad S_{n-1}(f )\in R_n[f ,\d f ,\dots, \d^{n-1}f ], \quad n\geq 1 .
\end{equation}

Formulas \eqref{JnA},  \eqref{eq.pfs} and \eqref{eq.pxs} yield:

\begin{corollary}\label{JnPn}
	For any $R_0$-algebra $A= R_0[\bx]/({\bf f})$ we have 
	\[
	J_nA \simeq R_n[\bx,\dots, \bx^{(n)}]/ (P_0({\bf f}),\dots, P_n({\bf f}))\simeq R_n[P_0(\bx), \dots, P_n(\bx)]/ (P_0({\bf f}),P_1({\bf f}),\dots,P_n({\bf f})).
	\]
\end{corollary}

We can now prove that  Buium's   $\pi$-jet algebras functor is left adjoint to the functor of $\pi$-typical Witt vectors.

\begin{theorem}\label{t.adjunction} 
	The functor $J_n$ from the category of $R_0$-algebras to the category of $R_n$-algebras is left adjoint to the functor $W_n:=W_{\pi,q,n}$.
\end{theorem}

\begin{proof}
	Let $B$ be an $R_n$-algebra and view $W_n(B)$ as an $R_0$-algebra  as in 
	Proposition \ref{proluniv3}. We first consider the case $A=R_0[\bx] =
	R_0[x_\gamma, \gamma \in \Gamma]$.
	For any  homomorphism of $R_0$-algebras
	\[
	\g\colon R_0[\bx]\longrightarrow W_n(B),\quad  x_\gamma\mapsto 
	(b_{\gamma,0},\dots,b_{\gamma,n})
	\] let $\Phi(\g)$ denote the homomorphism of $R_n$-algebras 
	\[\Phi(\g)\colon R_n[\bx,\dots, \bx^{(n)} ]=R_n[P_0(\bx), \dots, P_n(\bx)]\longrightarrow  B, \quad P_i(x_\gamma)\mapsto b_{\gamma,i}.
	\] 
	Then we have a natural bijection \[\Hom_{R_0}(R_0[\bx],W_n(B))\simeq \Hom_{R_n}(R_n[\bx,\dots, \bx^{(n)}], B), \quad \g\mapsto \Phi(\g).
	\]
	By construction the following diagram
	\begin{equation}\label{d.exppr}\xymatrix{
		R_0[\bx]  \ar[rr]^(0.3){\nexp{n}} \ar[drr]_\g && W_n(R_n[\bx,\dots, \bx^{(n)}  ])\ar[d]^{W_n(\Phi(\g))}\ar[rr]^{\text{pr}_i}&& \ar[d]_{\Phi(\g)} R_n[\bx,\dots, \bx^{(n)}] \\
		& & W_n(B) \ar[rr]^{\text{pr}_i} && B
	}\end{equation}
	commutes for all $0\leq i\leq n$, where the map $\text{pr}_i$ is the set-theoretic projection onto the $i$th component, i.e., $ \text{pr}_i(b_0,\dots,b_n)=b_i$. Note that the commutativity of the square on the right is trivial, while the commutativity of the triangle on the left depends on   $\nexp{n}(x_\gamma)=(P_0(x_\gamma),\dots, P_{n}(x_\gamma))$ \eqref{eq.expcomp},  and the definition $\Phi(\g)(P_i(x_\gamma))=b_{\gamma,i}$.

	Applying again \eqref{eq.expcomp}, one checks that $\g(f)=0$ if and only if $\phi(\g)(P_i(f))=0$ for any $0\leq i\leq n$. Together with Corollary \ref{JnPn} this fact says that  $\g$ in diagram \eqref{d.exppr} factors through 
	$R_0[\bx]/({\bold f})$ if and only if  $\Phi(\g)$ factors through  
	\[R_n[\bx,\dots, \bx^{(n)}]/(P_0({\bold f}), \dots,P_0({\bold f}))= R_n[P_0(\bx), \dots, P_n(\bx) ]/(P_0({\bold f}), \dots,P_0({\bold f})).\]
	
	Hence  for any $R_0$-algebra $A$ and any $R_n$-algebra $B$ there is a natural 
	bijection 
	\begin{equation}\label {eq.adj0}
	\Hom_{R_0}(A,W_n(B))\simeq \Hom_{R_n}(J_nA, B), \quad \g\mapsto \Phi(\g), \qquad \text{(where $W_n$ stands for $W_{\pi,q,n}$), }
	\end{equation}	 such that the diagram 
	\[\xymatrix{
		A  \ar[rr]^(0.4){\nexp{n}} \ar[drr]_\g && W_n(J_nA)\ar[d]^{W_n(\Phi(\g))}\ar[rr]^{\text{pr}_i}&& \ar[d]_{\Phi(\g)} J_nA \\
		& & W_n(B) \ar[rr]^{\text{pr}_i} && B
	}\]	
	commutes, and we are done. 
\end{proof}

\section{Jet spaces}
\label{Buim.js}
In this section, we wish to reconcile two apparently different approaches to $\pi$-jet spaces, one due to Buium and the other to Borger. When working with affine schemes we generalize both constructions to our context showing that they coincide. It is easy to see that Borger's construction extends naturally to non-affine schemes while Buium's construction requires additional assumptions,  e.g. nilpotency of $p$, and   one is led to work with formal schemes. After introducing both constructions we compare them  in Corollaries \ref{BoBua} \& \ref{cor.nm}. Our result  was known to Borger  in classical cases (cf. \cite[12.8]{bor11b}).

\subsection{Buium's $\pi$-jet spaces} 
In \cite{bui96} Buium develops the theory of $\pi$-jet spaces in the category of $\pi$-formal
schemes. However we want to talk about it in two distinct steps: first for the affine case and then for the non-affine case where we will be in the context
of $\pi$-formal schemes only.

Note that from \eqref{JnA}, one can naturally introduce the {\it Buium jet space} of an affine scheme  $X= \Spec A$ setting  
\[J^n_{Bu} X := \Spec J_nA.\]

Now we will pass to the $\pi$-formal category to follow Buium's original construction in \cite{bui96}.
Let $\hat\cO$ be the $\pi$-adic completion of $\cO_\mathfrak p$ and assume that our fixed prolongation sequence $R_*$ consist of  complete $\pi$-adic rings over $\hat \cO$. 
For an affine $\pi$-formal scheme $Y= \Spf A$, set $\Jbu Y = 
\Spf \widehat{J_nA}$ where $\widehat{J_nA}$ is the $\pi$-adic completion of
$J_nA$.
Now given a $\pi$-formal scheme $X$ over $R_0$ along with an open affine cover 
\[X = \bigcup_i V_i,
\]  
Buium in  \cite[p.315]{bui95} defines the $n$-th jet space 
$\Jbu{X}$ as 
\begin{equation}
\label{Jcover}
\Jbu X = \bigcup_i \Jbu {V_i},
\end{equation}
where the glueing data of   the spaces $\Jbu {V_i}$  naturally comes from that of the open subschemes $V_i$. 
The fundamental reason why such a glueing is possible is the fact that for $\pi$-adically complete algebras $A$ over $R_0$ one has the
following canonical isomorphisms of localized $R_n$-algebras
\begin{equation}
\label{local}
\widehat{J_n(A_s)} = \widehat{(J_nA)_s}, \mbox{ for all } s \in A\smallsetminus \mathrm{Nil}(A).
\end{equation}

Here one may ask the question of following a similar procedure as above for constructing an $n$-th jet space $\Jbu X$ for a general scheme $X$ over $R_0$. For any algebra $A$ (not necessarily $\pi$-adically complete) we have 
\begin{equation}
\label{non-local}
J_n(A_s) = J_n(A)_{s\phi(s)\cdots \phi^n(s)}, \mbox{ for all }  s \in A\smallsetminus \mathrm{Nil}(A).
\end{equation}
If $A$ is a $\pi$-adically complete algebra then (\ref{non-local}) leads to 
(\ref{local}). However this is not true in general for an $A$ that is not 
$\pi$-adically complete and hence this approach will not directly apply to the scheme context.

\subsection{The algebraic jet spaces}
\label{alg.js}
In this subsection we will develop some of the functorial theory of jet spaces over a fixed prolongation sequence $R_*$. Much of the notation is borrowed from \cite{bor11b}. One of the main goals of this section is to 
prove Theorem \ref{S03} and the proof is very much an adaption of the argument given for Theorem $12.1$ in \cite{bor11b}. However since our base prolongation sequence is not $\cO_*$ in general,  it is worthwhile to give some details.

For a fixed affine scheme $S=\Spec R$,
let  $\Aff{S}$  denote the category of affine schemes over $S$ equipped with the \'{e}tale topology and let  $\Sp{S}$ denote the category of sheaves of sets on $\Aff{S}$. Then we can consider the fully faithful embedding of the category of $S$-schemes, denoted by $\Sch{S}$, inside $\Sp{S}$. If $f\colon S'\to S$ is a morphism of affine schemes and $X$ a sheaf in $\Sp{S}$, we write $S'\times_S X$ for the pullback of $X$ to $\Sp{S'}$. When $R$ is a  ring decorated with sub- and superscripts we will  write $\Aff{R}$ in place of $\Aff{S} $ and  $\Sp{R}$ in place of  $\Sp{S}$.

Let $R_*=(R_0\to R_1\to \dots)$ be a fixed prolongation sequence and set  $\Sn=\Spec R_n$ and $\Snr m= \Sn \times_{\Spec \Ou} \Spec \cO/(\pi^m)$ its reduction modulo $\pi^m$; note that following Borger's notation  subscripts on rings become superscripts on schemes.

For any $n\geq 0$ and any sheaf $X$ in $\Sp{R_0}$ we    generalize Borger's construction introducing a  sheaf  $\Jbo X$ in $\Sp{R_n}$ given by 
\begin{equation*}
\Jbo X (B) := X(W_n(B))
\end{equation*}
for every $R_n$-algebra $B$, where $W_n(B)$ is here considered with its  $R_0$-algebra structure as in Corollary \ref{proluniv3}. When $R_*=\cO_*$, this sheaf is the one denoted by $W_{n*}(X)$ in \cite[\S 10.3]{bor11b}.
When both $X$ and $ \Jbo X$ are representable by schemes we may write
\begin{equation}\label{adjbx}
\Hom_{R_n}(\Spec B,\Jbo B) = \Hom_{R_0}(\Spec W_n(B),X).
\end{equation}

Theorem \ref{t.adjunction} immediately implies the following 

\begin{corollary}\label{BoBua}
	Let $X$ be an affine scheme over $R_0$. Then there is a natural isomorphism of sheaves
	\[\Jbo X\simeq J^n_{\enm{Bu}} X.\]
	In particular $\Jbo(X)$ is representable by an affine $R_n$-scheme.
\end{corollary}

Following \cite[\S 10.10]{bor11b} we say that a map $f\colon X\to Y$ in  $\Sp{R_0}$ is  {\it formally \'etale} if for all nilpotent closed immersions $\bar T\to T$ of affine schemes the induced map
\[X(T)\to X(\bar T)\times_{Y(\bar T)} Y (T)
\] is a bijection. Further $f$ is {\it locally of finite presentation} if for any cofiltered system
$ (T_i)_{i\in I}$ of affine schemes over $Y$,  the induced map $\colim_i \Hom_Y(T_i,X)\to \Hom_Y (\lim_i
T_i, X)$ is bijective. Finally $f$ is said to be \'etale if it is locally of finite presentation
and formally \'etale. For $f$ a morphism of schemes, these definitions are the usual ones. 

\begin{proposition}
	\label{pro.etale}
	Let $f\colon X \map Y$ be a 	formally \'{e}tale  (respectively locally of finite presentation, respectively \'etale) map in $\Sp{R_0}$. Then  the map $J^nf\colon \Jbo X\to \Jbo Y$  in $\Sp{R_n}$ has the same property. 
\end{proposition}
\begin{proof} The proof in \cite[11.1]{bor11b} works in this context also. We recall it briefly. Assume $f$ is formally \'etale. 	Let $\bar T=\Spec(B/I)\to T=\Spec(B)$ be a  closed immersion of affine $R_n$-schemes. Assume $I^m=0$ and note that $W_n(I)$ is a nilpotent ideal in $W_n(B)$ \cite[Corollary 6.4]{bor11a}. Then 
	\begin{multline*}\Jbo X(T)=X(W_n(B ))=X(W_n(B)/W_n(I))\times_{Y(W_n(B)/W_n(I))} Y(W_n(B))=\\
	X(W_n(B/I))\times_{Y(W_n(B/I))} Y(W_n(B))= \Jbo X(\bar T) \times_{\Jbo Y(\bar T)} \Jbo Y(T). 
	\end{multline*}
	Hence $\Jbo f$ is formally \'etale.

	Assume now $f$ is locally of finite presentation. We can repeat word by word the argument in \cite[11.1a)]{bor11b} setting $W_n^*T:=\Spec W_n(B)$ for any affine scheme $T=\Spec B$ over $Y$. The \'etale case follows then by definition of \'etaleness.
\end{proof}

\begin{corollary}\label{cor.open}
	Let $f\colon X\to Y$ be an open immersion of $R_0$-schemes. Then $\Jbo f\colon \Jbo X\to \Jbo Y$ is an open immersion of sheaves in $ \Sp{R_n}$, i.e., it is an \'etale monomorphism.  
\end{corollary}
\begin{proof}
	The functor $\Jbo \ $ preserves monomorphisms since it   is right adjoint to the functor  $\Spec(W_n(-))$ on affine $R_n$-schemes.    By the previous proposition it preserves \'etaleness. 
\end{proof}

\begin{lemma}\label{l.closed}
	Let $f\colon X\to Y$ be a closed immersion of $R_0$-schemes. Then $\Jbo f\colon \Jbo X\to \Jbo Y$ is a closed immersion of sheaves in $ \Sp{R_n}$, i.e., its base change along any map $T\to \Jbo Y$ with $T$ an $R_n$-scheme is a closed immersion of schemes.  
\end{lemma}
\begin{proof}
	As in the scheme theoretic case,
	it is sufficient to consider the case where $T$ is affine. 
	Let $T=\Spec C$ and consider a map $g\colon T\to \Jbo Y$, i.e.  $g\in \Jbo Y(C)$ and let $g'\colon \Spec W_n(C)\to Y$ be the corresponding map, i.e., the corresponding section in $Y(W_n(C))$. The map of functors $g$ factors as $T\to \Jbo (\Spec W_n(C)) \stackrel{\Jbo g'}{\to } \Jbo Y$ where the first map is the one associated to the identity on $\Spec W_n(C)$ via the adjunction \eqref{adjbx}. Hence it suffices to prove that the projection $\Jbo X\times_{\Jbo Y} \Jbo (\Spec W_n(C)) \to \Jbo (\Spec W_n(C))$ is a closed immersion of schemes.  Let $X'=X\times_Y\Spec W_n(C)$; it is an affine closed subscheme of $\Spec W_n(C)$ by hypothesis. Hence $\Jbo X'\to \Jbo (\Spec W_n(C))$ is a closed immersion by Corollary \ref{BoBua} and the explicit description of Buium jet spaces. Since the functor $\Jbo \ $ preserves fiber products, $\Jbo X'= \Jbo X\times_{\Jbo Y} \Jbo (\Spec W_n(C))$ and the proof is completed. 
\end{proof}

\subsection{Representability of algebraic jet spaces }
\label{sub.rep}
\begin{theorem}
	\label{S03}
	If $X$ is an $R_0$-scheme then $\Jbo X$ is an $R_n$-scheme for all $n$.
\end{theorem}
\begin{proof}
	We proceed by steps. Note that the fixed prolongation sequence $R_*$ lies above the constant prolongation $\cO_*$. 
	
	First assume that $X=Y\times_{\cO} S^{(0)}$, i.e. it is obtained by base change from  an $\cO$-scheme $Y$.  By definition 
	\[\Jbo X(B)=X(W_n(B))=Y(W_n(B))=\Jbo Y(B)= (W_{n*}( Y)\times_\cO \Sn)(B)
	\]  for any $R_n$-algebra $B$ where $W_n(B)$ is endowed with its 
	$R_0$-algebra structure and $W_{n*}(Y)$ is the absolute arithmetic jet space
	as in \cite{bor11b}.
	
	By \cite[Theorem 12.1]{bor11b}  $W_{n*}(Y)$ is represented by an $\cO$-scheme. 
	Hence $\Jbo X$ is represented by an $R_n$-scheme.
	
	In the general case, let $X$ be any $R_0$-scheme. In particular it is a $\cO$-scheme and we can consider  the closed immersion
	\[
	j\colon X\longrightarrow X':= X\times_\cO S^{(0)}=X\times_{R_0} (S^{(0)}\times_\cO S^{(0)})
	\] induced by the multiplication map $R_0\otimes_\cO  R_0\to R_0, a\otimes b\mapsto ab$.
	
	By the previous step  $\Jbo X'$ is representable by an $R_n$-scheme. Since $\Jbo(j)\colon \Jbo X\to \Jbo X'$ is a closed immersion by Lemma \ref{l.closed},  $\Jbo X$ is representable by an $R_n$-scheme as well. 
\end{proof}

The scheme $J^nX$ is called {\it algebraic $(\pi,q)$-jet space of level $n$}, or simply {\it jet space of level $n$, associated to} $X$.  

Let us use the same notation for the restriction $\Jbo \colon \mathrm{Sch}/{R_0}\to \mathrm{Sch}/{R_n}$ from the category of $R_0$-schemes to the category of $R_n$-schemes. This latter functor has a left adjoint, denoted by $W_n^*$,  which associates to any  $R_n$-scheme $Y$ the $W_n(R_n)$-scheme  $W_n^*(Y)$  constructed by gluing $\Spec W_n(A)$ as $\Spec A$ varies among the open subschemes of $Y$. 
For the construction of $W_n^*(Y)$ see also \cite{lz}, \cite{bor11b}.
Due to Corollary \ref{proluniv3} it is an $R_0$-scheme. 

\subsection{The $\pi$-fiber}

For a sheaf $X$ in $ \Sp{R_0}$ let $X^{(n)}$ denote its pullback $\Sn\times_{S^{(0)}}X$ to $ \Sp{R_n}$. Note that there is a canonical morphism of sheaves in  $ \Sp{R_n}$
\begin{equation}\label{jtox} \Jbo X   \to X^{(n)}\end{equation}
which is given by the maps $X(W_n(B))\to X(B)$ induced by the projection $W_n(B)\to B$ for every $R_n$-algebra $B$ (cf. the co-ghost map $\kappa_0$ in \cite[10.6.9]{bor11b}).

By \eqref{jtox} and functoriality of $J^n$ there exists a morphism
\begin{equation}\label{eq.jboxy} \Jbo X  \longrightarrow  \Jbo Y \times_{Y^{(n)}} X^{(n)}\end{equation}
in $\Sp{R_n}$. The following proposition states that this map becomes an isomorphism when working modulo powers of $\pi$.

\begin{proposition}
	\label{S01}
	Let $f\colon X \map Y$ be a map in $\Sp{R_0}$ which is 
	formally \'{e}tale.
	Then  the map 
	\[ \Sn_{m} \times_{\Sn} \Jbo X  \longrightarrow  \Sn_{m}
	\times_{\Sn} \Jbo Y \times_{Y^{(n)}} X^{(n)}\]
	is an isomorphism of sheaves over $\Sn_m$ for all $n$ and $m$.    
\end{proposition}
\begin{proof}  Let $B$ be an $(R_n/\pi^mR_n)$-algebra. Then   \eqref{eq.jboxy} gives the following canonical map
	\[ \Jbo X (B) = X(W_n(B)) \longrightarrow Y(W_n(B)) \times_{Y(B)} X(B) = 
	(\Jbo Y  \times_{Y^{(n)}} X^{(n)})(B)\]
	on $B$-sections. 
	Now we will construct an inverse to the above map. Given an element 
	$g \in Y(W_n(B)) \times_{Y(B)} X(B)$   this is determined by a pair of sections $(g', g'')$ making the following diagram commute 
	\[\xymatrix{
		\Spec B \ar@{^{(}->}[d] \ar[r]_{g'} & X \ar[d]^f \\
		\Spec W_n(B) \ar[r]_-{g''} & Y
	}\]
	Note that the left vertical arrow is a nilpotent thickening of $\Spec B$ since  $\pi$ is nilpotent in $B$ and hence the ideal $V(W_n(B))$ is nilpotent by \eqref{FV-pi} and \eqref{FV2}. 
	Since $f\colon  X \map Y$ is formally \'{e}tale, there is a unique lift 
	$\tilde{g}\colon  \Spec W_n(B) \map X$  of $g''$  and we are done.
\end{proof}

As a   consequence, one can get that on the level of formal schemes the jet functors respect open affine coverings.

\begin{corollary}\label{cor.nm}
	Let  $   X$ be an $S_0$-scheme. Then $ \Sn_m \times_{\Sn}\Jbo X$ is an $\Sn_m$-scheme for any $n$ and $m$.   Further if  $\{U_i\}_{i\in I}$ is an affine open covering of $X$  then  $ \{\Sn_m \times_{\Sn}\Jbo U_i\}_{i\in I}$ is an affine  open covering of  $ \Sn_m \times_{\Sn}\Jbo X$. As  consequences the formal $\pi$-completion of $\Jbo X$ is isomorphic to $\Jbu X $ and the special fiber of $\Jbo X$  is covered by the special fibers of the open subschemes $\Jbo U_i$'s.
\end{corollary}
\begin{proof}
	Representability by a scheme is clear by Theorem \ref{S03}.	By Corollaries \ref{cor.open} and \ref{BoBua} we know that the canonical map $ \Sn_m \times_{\Sn}\Jbo U_i \to \Sn_m\times_{\Sn}\Jbo X$  are open immersions with $  \Sn_m \times_{\Sn}\Jbo U_i$ affine schemes. Further,  the open immersions $U_i\times_X U_j\to U_i$ induce corresponding open immersions on the $n$-th jet spaces restricted to $\Sn_m$. We can then glue  the affine schemes $ \Sn_m \times_{\Sn}\Jbo U_i $ getting an $\Sn_m$-scheme $\bar J$. It remains to check that the induced map $\bar J\to \Sn_m \times_{\Sn}\Jbo X$ is an isomorphism. Given any section $s\in \Jbo(X)(B)=X(W_n(B))$ with $B$ an $R_0$-algebra, up to localization on $W_n(B)$ (or equivalently on $B$, since the corresponding affine spectra are homeomorphic) we may assume that $s$ comes from a section $s_i\colon \Spec (W_n(B))\to U_i$. Since $ U_i(W_n(B))=\Jbo(U_i)(B)\subset \bar J(B)$,  the section $s$ lies in $\bar J(B)$  and we are done.
\end{proof}

Let $X$ be an $R_0$-scheme, let $\ov{\Jbo X}$ denote the special fiber of  $\Jbo X$, i.e.,\[\ov{\Jbo X}:= \ov{\Jbo X}  \times_{\Spec \Ou} \Spec \cO/(\pi)=\ov{\Jbo X}  \times_{S^{(n)}}  S_1^{(n)}\] and let
$\ov{J_{Bu} X}$ denote the special fiber of the $\pi$-formal scheme $\Jbu X$. 

\begin{corollary}
	\label{BoBuSp}
	Let $X$ be an $R_0$-scheme. 	For all $n\geq 0$ we have $\ov{\Jbo X} \simeq \ov{J_{Bu}^n X}$.
\end{corollary}
\begin{proof}
	This follows from \eqref{Jcover} and Corollaries \ref{BoBua}  and \ref{cor.nm}.
\end{proof}

\section{Jet spaces and Greenberg transform}
In this section, we let $W(B) $ denote the ring of $p$-typical Witt vectors with coefficients in a ring $B$ (i.e, $W_{\pi,q}(B)$ with $\pi=p, q=p$). In addition, let $\fR$ be a fixed local artinian ring  with perfect residue field $k'$ of characteristic $p$. We recall the definition of the Greenberg algebra and the Greenberg transform.

\subsection{Greenberg algebra and Greenberg transform}
The {\it Greenberg algebra} associated to  $\fR$,  is the affine $k'$-ring scheme $\cR$ which represents the fpqc sheaf associated to the presheaf 
\[(\text{affine $k'$-schemes} )\to (\fR \text{-algebras}), \qquad \Spec(B)\mapsto W(B)\otimes_{W(k')} \fR ; \]
see \cite{gre, lip, bga} for an explicit description. There exists a canonical surjective homomorphism $\cR (B)\to W(B)\otimes_{W(k')} \fR$ which is an isomorphism if $RC$ is semiperfect, i.e., if the absolute Frobenius on $B$ is surjective \cite[Corollary A.2]{lip}.

\begin{example}\label{ex.w}
	If $\fR=W_m(k')$, then $\cR=\mathbb W_{m,k'}$ the $k'$-ring scheme of  $p$-typical  Witt vectors of   length  $m+1$, \cite[proof of Prop. A.1]{lip}. 
\end{example}

Let now $Z=\Spec(A)$ be an affine $\fR$-scheme and consider the functor 
\begin{equation}
\label{eq.furn}(\text{affine $k'$-schemes} )\to \text{({\rm Sets})}, \qquad \Spec(B)\mapsto Z(\cR(B)):=\Hom_\fR(\Spec(\cR(B)),Z). \end{equation}
By \cite{gre} and \cite[Proposition 6.4]{bga}, the above functor is representable by an affine scheme $\gr^{\fR}(Z)=\Spec(\gr^{\fR}(A))$, called {\it the Greenberg transform} or {\it Greenberg realization of $Z$} and hence
\begin{equation}\label{eq.adj1}
\Hom_{\fR}\big(\Spec \cR(B),\Spec(A)\big)\simeq\Hom_{ k'}\big(\Spec(B), \Spec(\gr^{\fR}(A))\big),
\end{equation}
for any $k'$-algebra $B$. In particular there is a natural bijection
\begin{equation}\label{eq.adj2}
\Hom_{\fR} \big(A,\cR(B)\big)\simeq\Hom_{ k'}\big( \gr^{\fR}(A), B\big).
\end{equation}

If $\fR=W_m(k')$ then the above adjunction gives for any $k'$-algebra $B$
\begin{equation}\label{eq.adj3}
\Hom_{W_m(k')} \big(A,W_m(B)\big)\simeq\Hom_{ k'}\big( \gr^{W_m(k')}(A), B\big).
\end{equation}

More generally, for any $k'$-scheme $Y$ one can glue $\Spec \cR(B)$ as $\Spec B$ varies among the open subschemes of $Y$ thus getting a scheme $h^\fR(Y)$ over $\fR$. The functor $h^\fR$ from the category of $k'$-schemes to the category of $\fR$-schemes admits a right adjoint $\gr^{\fR}$ so that there  are natural bijections
\begin{equation*}
\Hom_{\fR} \big(h^\fR Y,X\big)\simeq\Hom_{ k'}\big(Y, \gr^{\fR}(X)\big),
\end{equation*}
for any $k'$-scheme $Y$ and any $\fR$-scheme $X$; see \cite[\S 6]{bga}. By construction any affine open covering $\{U_i\}$ of $X$ induces an affine open covering  $\{\gr^{\fR}(U_i)\}$ of $\gr^{\fR}(X) $. 
\begin{example}\label{ex.w2}
	If    $\fR=W_m(k')$, then $h^\cR(Y)$ coincides with the scheme  $W_{m}^*(Y)$  introduced in \S \ref{sub.rep}. 
\end{example}

\subsection{Comparison between jet spaces and Greenberg transform}

Let $(\cO',\pi',q')$ be another triple with  $\cO'$ a Dedekind domain above $\cO$, $\pi'$ a generator of a fixed prime ideal $\mathfrak p'\subset \cO'$ such that $\pi\in (\pi')^e\cO'$, $q'=q^r$ the cardinality of the residue field $\cO'/\mathfrak p'$.
We write  $W(B)$ for the ring  of $p$-typical Witt vectors  with coefficients in $B$, and $W_{\pi,q}(B)$  for the $\pi$-typical variant constructed in Section \ref{s.pityp} from the data $(\cO,\pi,q)$. 
Similarly for the finite length rings $W_m(B)$ and $W_{\pi,q,m}(B)$.
For any $\cO'$-algebra $B$ there exists a natural homomorphism of $\cO$-algebras (cf. \cite[Proposition 1.2]{dri})
\begin{equation*}\label{eq.um0}
u\colon W_{\pi,q}(B)\to W_{\pi',q'}(B)
\end{equation*} 
determined by the conditions
\[u([b])=[b], \qquad u(F^r(x))=F(u(x)), \qquad u(V(x))=\frac{\pi}{\pi'}V(u(F^{r-1}(x))),\] 
for any $b\in B$ and $x\in W_{\pi,q}(B)$; here $F$ denotes Frobenius in both ring of Witt vectors and similarly for the Verschiebung $V$. Now, for any $\FF_q$-algebra $B$ there are  Drinfeld maps
\begin{equation}\label{eq.um}
u\colon W(B) \to W_{p,q}(B), \qquad  u_m\colon W_m(B)\to W_{p,q,m}(B),
\end{equation} 
which are isomorphisms if $B$ is perfect but not in general  \cite[1.2.1]{ff}; in fact, if $q=p^r, r\neq 1$ and $0\neq b\in B$ satisfies $ b^p=0$, then $u(V[b])=V[b^{p^{r-1}}]=0$.
In the general case  $e>1$, $u$ induces homomorphisms
\begin{equation}\label{eq.um1}
u_{m-1}\colon W_{\pi,q,m-1}(B)\to W_{\pi',q',me-1}(B).
\end{equation}

\subsection{Greenberg realization}
We now show that the special fiber of  a $\pi$-jet space is   very close to a Greenberg realization, thus generalizing \cite[Thm. 2.10]{bui96}.

Let  $k=\FF_q$  and  let $\cO=W(\FF_q)[\pi]/(\pi^e+\dots) $ be a finite totally ramified extension of $W(\FF_q)$ of degree $e$ with maximal ideal $\pi\cO$. Let $\tilde k$ be any  (possibly infinite) perfect field containing  $\FF_q$ and
let  $\tilde\cO=W_{\pi,q}(\tilde k)=W(\tilde k)[\pi]/(\pi^e+\dots)$  be the totally ramified extension of $W(\tilde k)$ of degree $e$ with uniformizer $\pi$. 
Let $\fR=\tilde \cO/\pi^{me}\tilde \cO$ for $m\geq 1$ and  write $\gr_{me-1}$ for the Greenberg realization $\gr^\fR$ (thus shifting indices with respect to notation in \cite[\S 7]{bga}).

\begin{theorem}\label{t.grj} Let $X$ be an $\fR$-scheme, $\gr_{me-1}(X)$ the Greenberg realization of   $X$ with respect to $\fR$ and  $J^{me-1}X $ the arithmetic $(\pi,q)$-jet space of level $me-1$ of $X$. There is a natural morphism \[v\colon \gr_{me-1}(X)\to \ov {J^{me-1} X}:=J^{me-1} X\times_{\tilde \cO} \Spec (\tilde k)\] of $\tilde k$-schemes which induces an isomorphism on inverse perfections, i.e., for any perfect $\tilde k$-algebra $B$ the map $ \gr_{me-1}(X)(B)\to J^{me-1} X(B) $ is an isomorphism. Further
	\begin{itemize}
		\item[i)] If  $\cO=\Z_p$ and $\fR=W_{m-1}(\tilde k)$, then $v$ is an isomorphism.
		\item[ii)] If $e=1$ and $m=1$ then   $v$ may be identified with the identity on $X\times_{\fR} \Spec \tilde k$.
	\end{itemize}
\end{theorem}
\begin{proof} 
	Since the Greenberg realization of $X$ is determined by the Greenberg realization of an affine open covering  of $X$ and the same holds for the special fiber of the jet space by Corollary \ref{cor.nm}, we may work locally on $X$ and  assume that $X=\Spec A$ is affine. Let $\gr_i(A)$ denote the ring of global sections of $\gr_i(X)$.
	Assume first that $\cO=W(\FF_q), \tilde \cO=W(\tilde k), \pi=p, e=1$, $\fR=W_{m-1}(\tilde k)$.
	By Example \ref{ex.w} we have $\cR=\WW_{m-1,\tilde k}$, the ring scheme of $p$-typical Witt vectors of length $m$ over $\tilde k$. 	
	Note that, since $\tilde k$ is perfect, Drinfeld's map \eqref{eq.um} gives an isomorphism $W(\tilde k)\simeq W_{p,q}(\tilde k)$. By \eqref{eq.adj3} and \eqref{eq.adj0} for $R_*=\tilde \cO_*$ we get
	\begin{multline*}
	\Hom_{ \tilde k}\big( \gr_{m-1}(A), B\big)\simeq \Hom_{W_{m-1}(\tilde k)} \big(A,W_{m-1}(B)\big)\simeq \\
	\Hom_{W(\tilde k)} \big(A,W_{m-1}(B)\big)\to 
	\Hom_{W(\tilde k)} \big(A,W_{p,q,m-1}(B)\big)\simeq \\ 
	\Hom_{W(\tilde k)}(J_{m-1}A, B)\simeq
	\Hom_{ \tilde k}(J_{m-1}A\otimes_{W_{m-1}(\tilde k)}\tilde k, B),\end{multline*}
	for any $\tilde k$-algebra $B$, where the arrow in the middle is induced by Drinfeld's map  in \eqref{eq.um} $u_{m-1}\colon W_{m-1}(B) \to W_{p,q,m-1}(B)$.
	Note that $u_0$ is the identity on $B$ and $u_{m-1}$, $m\geq 2$, is the identity if $q=p$.  Hence the identity map on $\gr_{m-1}(A)$ produces a natural homomorphism of $\tilde k$-algebras $J_{m-1}A\otimes_{W_{m-1}(\tilde k)}\tilde k \to \gr_{m-1}(A)  $ which is an isomorphism  if $q=p$ or $m=1$. This proves i) and ii) for $\cO=\Z_p$.  Further, since $u_{m-1}$ is an isomorphism for $B$ perfect, it induces a morphism of schemes $ v\colon \gr_{m-1}(X)\to J^{m-1}(X)\times_{\tilde\cO} \Spec (\tilde k)$ with the indicated property on inverse perfections.
	
	We now prove the theorem for   general $\cO$ and $\fR= \tilde \cO/\pi^{me}\tilde \cO$.  We first  introduce a generalization of   Drinfeld's map.
	Since $\fR=\tilde \cO/\pi^{me}\tilde \cO=\oplus_{i=0}^{e-1} W_{m-1}(\tilde k)\pi^i$ \cite[Remark 3.1b)]{bga2},  for any  $\tilde k$-algebra $B$ there is a natural   homomorphism of $\tilde \cO$-algebras
	\begin{eqnarray*}\tilde u\colon \cR(B)= W_{m-1}(B)\otimes_{W_{m-1}(\tilde k)} \fR=\oplus_{i=0}^{e-1} W_{m-1}(B)\pi^i &\to& W_{\pi,q,me-1}(B)  \\
		\sum_{i=0}^{e-1} w_i\otimes \pi^i&\mapsto& \sum_{i=0}^{e-1} u(w_i)\pi^i,\end{eqnarray*}
	where $w_i\in W_{m-1}(B)$ and  $u\colon W_{m-1}(B)=W_{p,p,m-1}(B)\to W_{\pi,q,me-1}(B)$ is Drinfeld's map in \eqref{eq.um1}. 
	By \eqref{eq.adj2} and \eqref{eq.adj0} with  $R_*=\tilde \cO_*$, for any $\tilde k$-algebra $B$ we have
	\begin{multline*}
	\Hom_{ \tilde k}\big( \gr_{me-1}(A), B\big)\simeq \Hom_{\fR} \big(A,\cR(B)\big)\simeq
	\Hom_{\tilde \cO} \big(A,\cR(B)\big)\to \\
	\Hom_{\tilde \cO}(A,W_{\pi,q,me-1}(B))\simeq 
	\Hom_{\tilde \cO}(J_{me-1}(A), B)\simeq \Hom_{\tilde k}(J_{me-1}(A)\otimes_{\tilde \cO} \tilde k, B),\end{multline*}
	where the arrow in the middle is induced by $\tilde u$. Hence we have a natural homomorphism $J_{me-1}(A)\otimes_{\tilde \cO} \tilde k\to \gr_{me-1}(A)  $, or in other words  there is a natural morphism $  \gr_{me-1}(X)\to \ov{ J^{me-1}(X)}$ of $\tilde k$-schemes.
	Since  $\tilde u$ maps $\cR(B)\simeq W_{m-1}(B)\otimes_{W(\FF_q)} \cO/\pi^{me}\cO$ isomorphically to $W_{\pi,q,me-1}(B)$,  when $B$ is perfect \cite[1.2.1]{ff}, the morphism $v$ is an isomorphism when passing to inverse perfection \cite[\S 5]{bga2}.   Further, since for $e=1$ and $m=1$ the map $\tilde u=u_0$ is the identity on $B$, assertion ii) is clear.  
\end{proof}

\begin{example}
	Let $X=\AA^1_\fR$. Then $\gr_{m}(\AA^1_{\fR})=\cR=\WW_{m-1,\tilde k}$, $J^{m-1}(X)= \WW_{p,q,m-1}$,   the $W(\tilde k)$-scheme of ramified Witt vectors. The map
	$v\colon \gr_{m-1}(\AA^1_\fR)\to J^{m-1}(X)\times_\fR \Spec \tilde k$ in Theorem \ref{t.grj}, evaluated on a $\tilde k$-algebra $B$, coincides with Drinfeld's map $u_{m-1}$ in \eqref{eq.um} and therefore is not an isomorphism in general.
\end{example}

The map $\tilde u$ in the proof of the previous theorem is  studied in details in \cite{bc}  when showing that the perfection of the Greenberg algebra of $\tilde \cO$ is the special fiber of the ring scheme of ramified Witt vectors.

Let $X$ be an $\mathfrak R$-scheme of dimension $r$. As noted in \cite[p.316]{bui95} one cannot expect   that, in the ramified case, the Greenberg transform and the special fiber of the  jet space  have the same dimension at all levels. However, Theorem \ref{t.grj} implies equality for a cofinal  subset of the set of levels.

By Corollary \ref{BoBuSp} the special fiber of $\Jbo X$ coincides with the special fiber of the formal jet space   $\widehat{J^n_{Bu}(X)}$. The above theorem states that, for  $X$ a $W(\tilde k)$-scheme, the Greenberg realization $\gr_{n}(X)$ is isomorphic to the special fiber of $\widehat{J^n_{Bu}(X)}$.  We can pass 
to the limit on $n$ obtaining the following 
\begin{corollary}\label{c.infty}
	Let $\tilde k$ be any perfect field of characteristic $p$,  $X$   a $W(\tilde k)$-scheme and $\gr(X)=\varprojlim \gr^{W_n(\tilde k)}(X)$. Then the $\tilde k$-scheme $\gr(X)$ is isomorphic to the special fiber of the $p$-jet space of infinite level $J^\infty(X)=\varprojlim J^n(X)$ or equivalently to the special fiber of $\widehat{J^\infty(X)}=\varprojlim \widehat{J^n(X)}$.
\end{corollary}  
Note that the above result generalizes the one in \cite[Theorem 2.10]{bui96} removing  the smoothness  hypothesis on $X$ as well as the condition that $\tilde k$ is algebraically closed.

\footnotesize{
	
}
\end{document}